\documentclass[11pt]{amsart}
\usepackage{amssymb}
\usepackage{amstext}
\usepackage{amscd}
\usepackage{comment}
\usepackage{color}
\usepackage[dvipsnames]{xcolor}
\usepackage{url}
\usepackage{hyperref}
\usepackage{graphicx}
\usepackage{caption}
\usepackage{subcaption}

\newtheorem{teorema}{Theorem}[section]
\newtheorem*{teo}{Theorem}
\newtheorem*{coro*}{Corollary}
\newtheorem*{propo}{Proposition}

\newtheorem*{thm}{Theorem}
\newtheorem{coro}[teorema]{Corollary}
\newtheorem{lema}[teorema]{Lemma}

\newenvironment{demode}
  {\noindent {{\it Proof of }}}%
  {\par \hfill \fbox{}}

\newtheorem*{defiRO}{Definition (Class ($\mathcal{RO}$))}

\theoremstyle{definition}
\newtheorem{ej}[teorema]{Example}

\newtheorem{defi}[teorema]{Definition}

\newcommand{\pe}[1]{\langle#1\rangle}
\newcommand{\EL}{\mathcal{L}}

\newcommand{\C}{\mathbb{C}}
\newcommand{\N}{\mathbb{N}}
\newcommand{\inte}{\textnormal{int}}
\newcommand{\D}{\mathbb{D}}
\newcommand{\T}{\mathbb{T}}

\newcommand{\al}{\alpha}

\newcommand{\PR}{\textnormal{Re}}
\newcommand{\sumn}{\sum_{n=1}^{\infty}}
\newcommand{\sumno}{\sum_{n=0}^{\infty}}

\newcommand{\sumk}{\sum_{k=1}^N}

\newcommand{\R}{\mathbb{R}}
\newcommand{\RO}{\mathcal{RO}}

\newcommand{\x}{\textnormal{\textbf{x}}}

\newcommand{\Rpart}{\textnormal{Re}}
\newcommand{\Impart}{\textnormal{Im}}

\begin{document}
	\title[Hyperinvariant subspaces]%{For Finite-Rank Perturbations Of Diagonal Operators}
{Finite rank perturbations of normal operators:\\ Hyperinvariant subspaces and a problem of Pearcy}	
%   Information for first author
	\author{Eva A. Gallardo-Guti\'{e}rrez}
	\address{Eva A. Gallardo-Guti\'errez and F. Javier Gonz\'alez-Doña \newline
		Departamento de An\'alisis Matem\'atico y Matem\'atica Aplicada,\newline
		Facultad de Matem\'aticas,
		\newline Universidad Complutense de
		Madrid, \newline
		Plaza de Ciencias N$^{\underbar{\Tiny o}}$ 3, 28040 Madrid,  Spain
		\newline
		and Instituto de Ciencias Matem\'aticas ICMAT (CSIC-UAM-UC3M-UCM),
		\newline Madrid,  Spain } \email{eva.gallardo@mat.ucm.es}

	%   Information for scond author
	\author{F. Javier Gonz\'alez-Doña}
	\email{javier.gonzalez@icmat.es}
	
	\thanks{Both authors are partially supported by Plan Nacional  I+D grant no. PID2019-105979GB-I00, Spain,
		the Spanish Ministry of Science and Innovation, through the ``Severo Ochoa Programme for Centres of Excellence in R\&D'' (CEX2019-000904-S) and from the Spanish National Research Council, through the ``Ayuda extraordinaria a Centros de Excelencia Severo Ochoa'' (20205CEX001) and Grupo UCM ref. 910346. \newline
		Second author also acknowledges support of the Grant SEV-2015-0554-18-3 funded by: MCIN/AEI/ 10.13039/501100011033.}
	


\date{July 2022, revised January 2023}

	\begin{abstract}
 Finite rank perturbations of diagonalizable normal operators acting boundedly on infinite dimensional, separable, complex Hilbert spaces are considered from the standpoint of view of the existence of invariant subspaces. In particular, if $T=D_\Lambda+u\otimes v$ is a rank-one perturbation of a diagonalizable normal operator $D_\Lambda$ with respect to a basis $\mathcal{E}=\{e_n\}_{n\geq 1}$ and the vectors $u$ and $v$ have Fourier coefficients $\{\alpha_n\}_{n\geq 1}$ and $\{\beta_n\}_{n\geq 1}$ with respect to $\mathcal{E}$ respectively,  it is shown that $T$ has non trivial closed invariant subspaces provided that either $u$ or $v$ have a Fourier coefficient which is zero or $u$ and $v$ have non zero Fourier coefficients and
 $$ \sum_{n\geq 1} |\al_n|^2 \log \frac{1}{|\al_n|} + |\beta_n|^2 \log \frac{1}{|\beta_n|}  < \infty.$$
 As a consequence, if $(p,q)\in (0,2]\times (0,2]$ are such $\sumn (|\al_n|^p + |\beta_n|^q )< \infty,$ it is shown the existence of non trivial closed invariant subspaces of $T$ whenever
 $$(p,q)\in (0,2]\times (0,2]\setminus \{(2, r), (r, 2):\; r\in(1,2]\}.$$
 Moreover, such operators $T$ have non trivial closed  hyperinvariant subspaces whenever they are not a scalar multiple of the identity. Likewise, analogous results hold for finite rank perturbations of $D_\Lambda$.  This improves considerably previous theorems of Foia\c{s}, Jung, Ko and Pearcy \cite{FJKP07}, Fang and Xia \cite{FX12} and the authors \cite{GG} on an open question explicitly posed by Pearcy in the seventies.
 \end{abstract}

	\maketitle

\section{Introduction and preliminaries}

Despite its simplicity, apparently one of the most difficult questions in the theory of invariant subspaces in separable, infinite dimensional complex Hilbert spaces $H$ is the problem of the existence of non trivial closed invariant subspaces for a compact perturbation of a selfadjoint operator. Livsi\u{c} solved this problem for nuclear perturbations, Sahnovi\u{c} for Hilbert-Schmidt perturbations, and Gohberg and Krein, Macaev, and Schwartz for the perturbation being in the Schatten von Neumann class $\mathcal{C}_p$, $1\leq p<\infty$ (see \cite{Dunford and Schwarz} for more on the subject). In 1992, Lomonosov \cite{Lo} proved the existence of real invariant subspaces for compact perturbations of selfadjoint operators, but it is still an open question if every compact perturbation of a selfadjoint operator has a non trivial closed invariant subspace.

\smallskip

The situation is still even hopeless if one considers compact perturbations of a bit broader class of operators, namely \emph{normal operators}. In such a case, if a bounded linear operator $T$ acting on $H$  is the sum of a normal operator $N$ and a compact operator $K$, the results ensuring the existence of non trivial closed invariant subspaces of $T$ have been mostly restricted to cases when the spectrum of $N$ is contained in a Jordan curve $\gamma$ of certain regularity and $K$ satisfies additional compactness criteria (like belonging to the Schatten von Neumann class $\mathcal{C}_p$) (see the pioneering work \cite{Kitano},  or the works \cite{Radjabalipour-Radjavi}, \cite{Radjabalipour} and \cite{Chalendar2} as well as  the references therein).

\smallskip

\noindent In this general framework, in 1975 Pearcy posed the following explicit problem (see \cite[Problem K]{Pearcy})

\medskip

\begin{quotation}
\emph{Suppose $N$ is a diagonal normal operator whose eigenvalues constitute a dense subset of the unit disc $\mathbb{D}$. Show that every operator of the
form $N + F$ has a non trivial invariant subspace, where $F$ is an operator of rank one.}
\end{quotation}

\medskip

In 2007, motivated by such a question, Foia\c{s}, Jung, Ko and Pearcy undertook a significant study of what they claimed is \emph{a stubbornly intractable problem}. In particular, they proved in  \cite{FJKP07} that there is a large class of  rank-one perturbations of diagonal operators having, actually, non trivial closed \emph{hyperinvariant subspaces}, that is, closed subspaces which are invariant under every operator in the commutant of $T$. By recalling that any rank-one operator can be expressed by $u\otimes v$, where $u, v$ are non zero vectors in $H$ and
$$u\otimes v(x) = \pe{x,v}u,   \quad \text{ for } x \in H,$$
they proved the following:

\smallskip

\begin{teo}[Foias, Jung, Ko and Pearcy, \cite{FJKP07}] Let $H$ be an infinite dimensional, separable, complex Hilbert space and $\mathcal{E} =\{e_n\}_{n\geq 1}$ and orthonormal basis of $H$. Let $u= \sumn \al_ne_n$ and $v = \sumn \beta_ne_n$ be  non zero vectors in $H$. Suppose $\Lambda=\{\lambda_n\}_{n\geq 1}$ is any bounded sequence of the complex plane and $D_{\Lambda}$ the diagonal operator with respect to $\mathcal{E}$ associated to $\Lambda$, namely, $D_{\Lambda} e_n=\lambda_n e_n$ for $n\geq 1$. Then, the rank-one perturbation of $D_{\Lambda}$
$$T=D_\Lambda +u\otimes v$$
has  non trivial closed invariant subspaces provided that
\begin{equation}\label{FJKP condition}
\sumn |\al_n|^{2/3}+|\beta_n|^{2/3} < \infty.
\end{equation}
Moreover, $T$ has non trivial closed hyperinvariant subspaces whenever it is not a scalar multiple of the identity.
\end{teo}
	
In 2012, Fang and Xia were able to strengthen the techniques of Foia\c{s}, Jung, Ko and Pearcy  also for finite rank perturbations of $D_\Lambda$, proving that the exponent ``2/3'' in \eqref{FJKP condition} could be pushed further up to 1. In particular, their result for rank-one perturbations of diagonal operators states the following:

\begin{teo}[Fang and Xia, \cite{FX12}] With the notation as introduced above, the linear bounded operator $T = D_\Lambda + u\otimes v$ has non trivial closed invariant subspaces whenever
\begin{equation}\label{equation FX}
\sumn \left( |\alpha_n|+ |\beta_n| \right )< \infty.
\end{equation}
Moreover, $T$ has non trivial closed hyperinvariant subspaces whenever it is not a scalar multiple of the identity.
\end{teo}

\begin{figure}[htb]
\centering
\begin{minipage}{.5\textwidth}
  \centering
  \includegraphics[width=.7\linewidth]{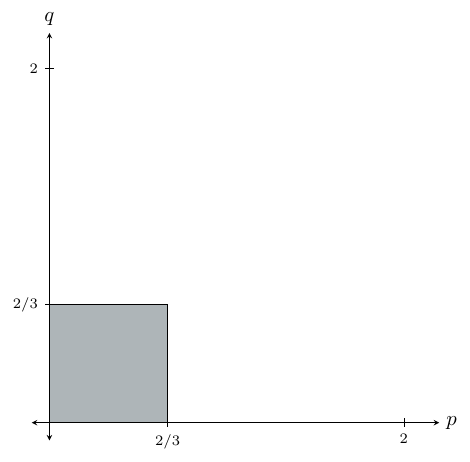}
 \captionof{figure}{ }
  \label{fig:1}
\end{minipage}%
\begin{minipage}{.5\textwidth}
  \centering
  \includegraphics[width=.68\linewidth]{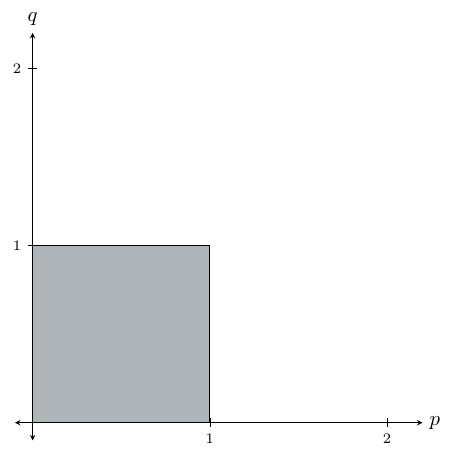}
  \captionof{figure}{ }
  \label{fig:2}
\end{minipage}
\vspace*{-0,2cm}
\end{figure}

Recently, the authors have been able to describe spectral subspaces associated to closed sets of the complex plane of finite rank perturbations of diagonal operators and, in particular, of normal operators of multiplicity one whose eigenvectors span $H$. Such a description has allowed them to exhibit proper closed hyperinvariant subspaces
as far as such spectral subspaces are both non zero and non-dense. As a consequence, they have been able to enlarge  the class of finite rank perturbations of a diagonalizable normal operator which are known to have non trivial closed hyperinvariant subspaces. In particular, they proved the following in the case of  rank-one perturbations of diagonal operators:

\begin{teo}[Gallardo-Gutiérrez and González-Doña, \cite{GG}] With the notation as introduced above, the linear bounded operator $T = D_\Lambda + u\otimes v$ has non trivial closed invariant subspaces provided that either
$$\sumn |\al_n|< \infty$$
or
$$\sumn |\beta_n| < \infty.$$
Moreover, $T$ has non trivial closed hyperinvariant subspaces whenever it is not a scalar multiple of the identity. \label{gg}
\end{teo}

\begin{figure}[htb]
\centering
  \includegraphics[width=.36\linewidth]{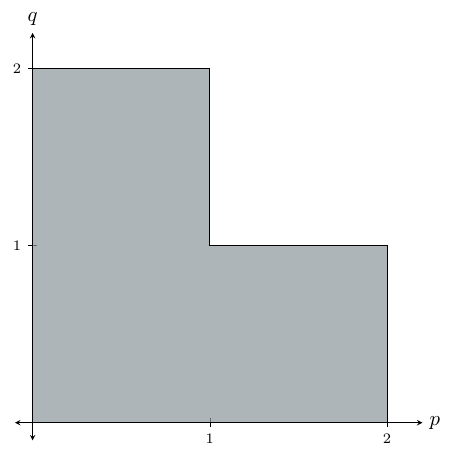}
  \captionof{figure}{ }
  \label{fig:3}
\end{figure}

By considering the pairs $(p,q)\in (0,2]\times (0,2]$ such that the condition
$$
\sumn (|\al_n|^p + |\beta_n|^q )< \infty,
$$
implies the existence of non trivial closed hyperinvariant subspaces of the non-scalar operator $T=D_\Lambda + u\otimes v$, Foias, Jung, Ko and Pearcy Theorem states that this is the case when $(p,q)\in  (0,2/3]\times (0,2/3]$, Fang and Xia Theorem if  $(p,q)\in (0,1]\times (0,1]$ and the authors' theorem  whenever $(p,q)\in (0,2]\times (0,1] \cup (0,1]\times (0,2]$ (see the Figures \ref{fig:1}, \ref{fig:2} and \ref{fig:3}).

\medskip

The goal of the present manuscript is taking further the previous results and proving, in particular, the following result:

\medskip

\begin{thm}[Theorem \ref{main result}] With the notation as introduced above, the linear bounded operator $T = D_\Lambda + u\otimes v$ has non trivial closed invariant subspaces provided that either $u$ or $v$ have a Fourier coefficient which is zero or $u$ and $v$ have non zero Fourier coefficients and
\begin{equation}\label{sumabilidad-intro}
    \sumn |\al_n|^2 \log \frac{1}{|\al_n|} + |\beta_n|^2 \log \frac{1}{|\beta_n|}< \infty.
\end{equation}
Moreover, if $T$ is not a scalar multiple of the identity, it has non trivial closed hyperinvariant subspaces.
\end{thm}

It is worthy to remark that in the case that $u$ or $v$ have a zero Fourier coefficient, Ionascu \cite{I01} proved that either $T$ or its adjoint $T^*$ have an eigenvalue. So the relevant part of the statement is when $u$ and $v$ have non zero Fourier coefficients. In such a case, the summability condition \eqref{sumabilidad-intro} along with authors' theorem yields, in particular, that the series $\sumn (|\al_n|^p + |\beta_n|^q )$
converges for every
$$(p,q)\in (0,2]\times (0,2]\setminus \{(2, r), (r, 2):\; r\in(1,2]\}.$$

\begin{figure}[htb]
\centering
  \includegraphics[width=.36\linewidth]{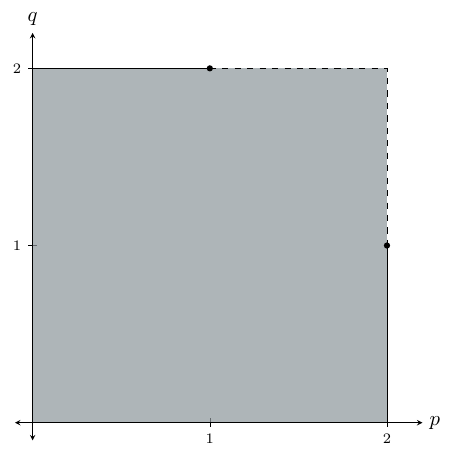}
  \captionof{figure}{ }
  \label{fig:4}
\end{figure}

Accordingly, as a consequence, if $T = D_\Lambda + u\otimes v$ is non-scalar, it has non trivial closed hyperinvariant subspaces as far as
$$(p,q)\not \in \{(2, r), (r, 2):\; r\in(1,2]\}$$
(see Figure \ref{fig:4}). Clearly, this enlarges considerably the class of rank-one perturbations of  diagona\-lizable normal operators which are known to have non trivial closed hyperinvariant subspaces.

\medskip

In order to prove Theorem \ref{main result}, we strengthen the methods in \cite{GG} by exhibiting non zero and non-dense spectral subspaces in Section \ref{sec2}. A significant role in our approach here will be also played by \emph{quasisimilar operators} and Theorem 2.5 in the work \cite{FJKP08}, where the authors proved that the commutants of rank-one perturbations of a diagonalizable normal operators are abelian.

\medskip

Likewise, in Section \ref{sec3} we will present an example showing the threshold of our approach which, in particular, is not covering those pairs $(p,q)\in \{(2, r), (r, 2):\; r\in(1,2]\}$. % leaving the following question open:

%\begin{quotation}
%Let $H$ be an infinite dimensional, separable, complex Hilbert space, $\mathcal{E} =\{e_n\}_{n\geq 1}$ and orthonormal basis and $u= \sumn \al_ne_n,$  $v = \sumn \beta_ne_n$ non zero vectors in $H$ such that
%\begin{equation*}
%\sumn |\al_n|^{p}+|\beta_n|^{q} < \infty,
%\end{equation*}
%for $(p,q)\in \{(2, r), (r, 2):\; r\in(1,2]\}$. If $\Lambda$ is any bounded sequence of the complex plane, does the rank-one perturbation of $D_{\Lambda}$
%$$T=D_\Lambda +u\otimes v$$
%have  non trivial closed invariant subspaces?
%\end{quotation}

\medskip

Finally, in Section \ref{sec4} we deal with finite rank perturbations of diagonal operators proving Theorem \ref{resultado rango finito}. Some of the results previously proved for rank-one perturbations of diagonal operators are generalized to this setting, pointing out the main difficulties. This result generalizes Fang and Xia Theorem \cite[Theorem 1.4]{FX12} and \cite[Theorem 6.4]{GG} and involves the analysis of an \emph{ad-hoc} matrix associated to $T$.

\medskip

We close this introductory section by recalling some preliminaries and relevant results addressed in the aforementioned works which will be of help to follow the present manuscript.

%\medskip

\subsection{The framework} In what follows, $H$ will denote an infinite dimensional separable complex  Hilbert space, $\EL(H)$ the Banach algebra of all bounded linear operators on $H$ and $\mathcal{E}= \{e_n\}_{n\geq 1}$ an (ordered) orthonormal basis of $H$ fixed.

If $\Lambda = (\lambda_n)_{n\geq 1}$ is any bounded sequence in the complex plane $\C$,  the diagonal operator $D_\Lambda$ with respect to $\mathcal{E}$ associated to $\Lambda$ is defined by
$$D_\Lambda e_n = \lambda_n e_n,\qquad (n\geq 1).$$
Any rank-one perturbation of $D_\Lambda$  can be written as
\begin{equation}\label{diagonal expression}
	T = D_\Lambda + u\otimes v,
\end{equation}
where $u,$ and $v$ are non zero vectors in $H$. It is important to note that \eqref{diagonal expression} is not unique as far as rank-one perturbations of diagonal operators concern but, if $u$ and $v$ have non zero components for every $n\geq 1$, uniqueness in \eqref{diagonal expression} follows (see \cite[Proposition 1.1]{I01}).

Note that rank-one perturbations of normal operators whose eigenvectors span $H$ are unitarily equivalent to those expressed by \eqref{diagonal expression}, and consequently, it suffices to study them from the standpoint of invariant subspaces.

\subsection{$(\RO)$ operators} In \cite{FJKP07}, the authors introduced the class of operators $(\RO)$ as follows:

\smallskip

\begin{defiRO} Let $\Lambda=(\lambda_n)_{n\geq 1} \subset \C$ be a bounded sequence and $D_{\Lambda}$ the diagonal operator associated to $\Lambda$ with respect to $\mathcal{E}$.  The rank-one perturbation of $D_{\Lambda}$
$$T = D_{\Lambda}+u\otimes v,$$
where $u = \sumn \al_n e_n$ and $v = \sumn \beta_n e_n$ are non zero vectors in $H$, belongs to the class $(\RO)$ if:
\begin{enumerate}
		\item [(i)] $\al_n\beta_n \neq 0$ for every $n \in \N$;
		\item [(ii)] the map $n \in \N \mapsto \lambda_n \in \Lambda$ is injective;
		\item [(iii)] the derived set $\Lambda'$ is not a singleton.
\end{enumerate}
\end{defiRO}

\noindent In particular, they proved that rank-one perturbations of diagonal operators not belonging to $(\RO)$  have always non trivial closed hyperinvariant subspaces as far as they are non-scalar (see \cite[Proposition 2.1]{FJKP07}).

\subsubsection{Spectrum of $(\RO)$ operators} Both the spectrum $\sigma(T)$ and the point spectrum $\sigma_p(T)$ of operators $T$ belonging to the class $(\RO)$ were characterized by Ionascu \cite{I01} in terms of Borel series. More precisely,  let $T= D_{\Lambda}+u\otimes v \in (\RO)$, where $\Lambda=(\lambda_n)_{n\geq 1}$ is a bounded sequence in $\mathbb{C}$,  $u = \sumn \al_n e_n$ and $v = \sumn \beta_n e_n$ are non zero vectors in $H$. Denote by $f_T$ the Borel series associated to $T$, namely
\begin{equation}\label{Borel series T}
f_T(z) = \sumn \frac{\al_n\overline{\beta_n}}{\lambda_n-z},
\end{equation}
for those $z\in \mathbb{C}$ such that the series converges. Clearly, $f_T$ is an analytic function in $\C\setminus \overline{\Lambda}$ and next result states that, in some sense, it encodes the point spectrum of $T$.

\begin{thm}[Ionascu \cite{I01}] With the introduced notation, $z \in \C$ belongs to $\sigma_p(T)$ if and only if
\begin{enumerate}
		\item [(i)] $z\not \in \Lambda$,
		\item [(ii)] $u \in \textnormal{ran}(D_\Lambda-z I),$ or equivalently
$$\sumn \frac{|\al_n|^2}{|z-\lambda_n|^2}<\infty,$$
		\item [(iii)]$f_T(z)+1=0.$
\end{enumerate}
Moreover, the spectrum of $T$ is given by
\begin{equation}\label{autovalores}
\sigma(T) = \Lambda' \cup \left \{z \in \C\setminus \overline{\Lambda}: f_T(z)+1=0 \right \}.
\end{equation}
\end{thm}

%Roughly speaking, the zeros of the function $z\mapsto 1+f_T(z)$ encode the behaviour of the point spectrum of $T$, under the assumption of $u \in \textnormal{ran}(D_\Lambda-zI).$ As we will see, this function also defines the behaviour of certain local spectral properties, under some similar (and weaker) assumptions. \\

\medskip

\subsection{Spectral subspaces of $(\RO)$ operators} In order to prove both Theorem \ref{main result} and Theorem \ref{resultado rango finito}, we will exhibit non zero and non-dense spectral subspaces associated to certain closed sets in $\mathbb{C}$. Spectral subspaces arise, in a natural way, scanning the possible decompositions of the spectrum of an operator and play a role in the \emph{Local Spectral Theory}. For the sake of completeness, we collect some of the basic tools and refer to  the monograph \cite{LN00} for more on the subject.

\smallskip

Recall that an operator $T\in \mathcal{L}(H)$ has the \emph{single-valued extension property} (SVEP) if for every connected open set $G\subset \mathbb{C}$ and every analytic function $f:G\to X$  such that $(T-\lambda I) f(\lambda)\equiv 0$ on $G$, one has $f\equiv 0$ on $G$. Note that every operator $T$ such that $\sigma_p(T)$ has empty interior has the SVEP.

\medskip

If $T$ has the SVEP and $x\in H$, the \textit{local resolvent} $\rho_T(x)$ of $T$ at $x$  is the union of all open sets $U \subset \C$ such that there exists a vector-valued holomorphic function $f_x: U \rightarrow H$ such that
\begin{equation}\label{ecuacion resolvente local}
(T-zI)f_x(z) = x \quad (z\in U).
\end{equation}
The \textit{local spectrum} $\sigma_T(x)$ at $x$ is defined as the complement of $\rho_T(x)$:
$$\sigma_T(x) = \C \setminus \rho_T(x).$$
Among many other properties, the local spectrum satisfies $\sigma_T(x) \subseteq \sigma(T)$ for every $x \in X$ and $\sigma_T(0) = \emptyset$.

\medskip

For every subset $\Omega\subset \mathbb{C}$, \emph{the local spectral subspace of $T$} associated with $\Omega$ is
$$H_T (\Omega)= \{x \in H :\; \sigma_T(x) \subseteq  \Omega\}.$$
If $\Omega_1\subset \Omega_2$ then $H_T (\Omega_1)\subset H_T (\Omega_2)$ and $H_T(\Omega) = H_T(\Omega\cap \sigma(T))$.
It is important to note that $H_T(\Omega)$ is a linear manifold that is hyperinvariant for $T$ but not necessarily closed even for closed subsets (indeed, even in the case that $T$ has the SVEP, see \cite[Chapter 2]{AIENA-book} for instance). The following proposition deals with spectral subspaces associated to closed sets (see \cite[Proposition 2.5.1]{LN00}):

\begin{propo}
Assume $T\in \EL(H)$ and both $T$ and its adjoint $T^*$ have the SVEP. Let $F_1, F_2\subset \C$ be disjoint closed sets. Then
\begin{enumerate}
\item[(1)] $H_T(F_1 \cup F_2) = H_T(F_1) \oplus H_T(F_2)$ holds as an algebraic sum.
\item[(2)]$H_T(F_1) \subseteq H_{T^*}(F_2^*)^\perp,$  where $F_2^*=\{\overline{z}:\; z\in F_2\}$.
\end{enumerate}
\end{propo}
\noindent Note that the difference of the statement (2) and  Proposition 2.5.1 in \cite{LN00} is due to the con\-si\-deration of the Hilbert adjoint operator $T^*$. As a consequence, we have the following

\begin{coro}\label{adjuntos}
Let $T\in \EL(H)$ be such that both $T$ and $T^*$ have the SVEP. Suppose $F_1$ and $F_2$ are disjoint closed sets in $\C$ such that $H_T(F_1)$ and $H_{T^*}(F_2^*)$ are non zero linear manifolds. Then $\overline{H_T(F_1)}$ is a non trivial closed hyperinvariant subspace for $T$.
\end{coro}

We end this section by stating a result which ensures the existence of non zero vectors belonging to some particular spectral subspaces of operators in the class $(\RO)$.

\begin{thm}[Theorem 3.2, \cite{GG}] Let $\Lambda=(\lambda_n)\subset \mathbb{C}$ be a bounded sequence, $u = \sumn \al_ne_n$, $v= \sumn \beta_n e_n$ non zero vectors in $H$ and $T=D_\Lambda + u \otimes v$ a bounded linear operator in the class $(\RO)$. Assume $\sigma(T)$ is connected and both $\sigma_p(T)$ and $\sigma_p(T^*)$ are empty. Suppose further that there exists a closed, simple, piecewise differentiable curve $\gamma$ in $\C$ not intersecting $\Lambda$ such that
\begin{enumerate}
    \item [(i)] $\sigma(T)\cap \inte(\gamma) \neq \emptyset$.
    \item [(ii)] The map $$ \xi \in \gamma \mapsto \frac{1}{1+f_T(\xi)}$$ is well defined and continuous on $\gamma.$
    \item [(iii)] $$ \sumn \left( \int_\gamma \frac{|d\xi|}{|\lambda_n-\xi|}\right)^2|\al_n|^2 < \infty.$$
\end{enumerate}
Then, the spectral subspace $H_T(\overline{\inte(\gamma)})$ is non zero.
\end{thm}

\medskip

In condition (iii), $|d\xi|$ denotes the arc length measure.

\section{Rank-one perturbations of diagonalizable normal operators} \label{sec2}

As pointed in the Introduction, the aim of this section is to prove the following theorem:

\begin{teorema}\label{main result}
Let $T = D_\Lambda + u\otimes v \in \EL(H)\setminus \mathbb{C}\,  Id_{H}$ be any rank-one perturbation of a diagonal normal operator $D_{\Lambda}$ with respect to an orthonormal basis $\mathcal{E}=\{e_n\}_{n\geq 1}$ where $u = \sumn \al_n e_n$ and $v = \sumn \beta_n e_n$. Then $T$ has non trivial closed hyperinvariant subspaces provided that either $u$ or $v$ have a Fourier coefficient which is zero or $u$ and $v$ have non zero Fourier coefficients and
\begin{equation}\label{sumabilidad}
 \sumn |\al_n|^2 \log \frac{1}{|\al_n|} + |\beta_n|^2 \log \frac{1}{|\beta_n|}<\infty.
\end{equation}
Moreover, if $T\in (\RO)$, $\sigma(T)$ is connected and $\sigma_p(T)\cup\sigma_p(T^*) = \emptyset$, it follows that $T$ does have non zero spectral subspaces which are not dense.
\end{teorema}

In order to prove Theorem \ref{main result}, we require some preliminary work. First, we consider the following lemma which generalizes Lemma 2.1 in \cite{FX12}:

\begin{lema}\label{lema fx}
Let $\Lambda = (\lambda_n)_n \subset \C$ be any bounded sequence of distinct complex numbers and $\{\al_n\}_{n\geq 1} \subset \C\setminus \{0\}$ a sequence  such that
\begin{equation}\label{serie logaritmo}
 \sumn  |\al_n|^2\, \log \frac{1}{|\al_n|} < \infty.
\end{equation} Then, for almost every $x \in \R\setminus\{ \PR(\lambda_n) : n \in \N\}$
 \begin{equation}\label{lema 2}
    \sumn \frac{|\al_n|^2}{|\PR(\lambda_n)-x|} < \infty.
\end{equation}
\end{lema}

Though the proof is inspired by that of \cite[Lemma 2.1]{FX12}, the conclusion is strengthened.

\begin{proof} Without loss of generality, we may assume that the clousure  $\overline{\Lambda}$ is contained in the unit disc $\D$. Accordingly, the series in \eqref{lema 2} converges for every $|x|> 1.$ Denote by $\Theta$ the set of $x \in [-1,1]\setminus \{\PR(\lambda_n): n\in \N\}$ such that
$$\sumn \frac{|\al_n|^2}{|\PR(\lambda_n)-x|}$$
diverges. We will show that the Lebesgue measure of $\Theta$ is zero.%$m(\Theta) = 0.$

Fix $0<\varepsilon< 1$ and let $\delta > 0$ be so that
$$ 2\delta \sumn  |\al_n|^2< \varepsilon.$$
For each $n \in \N$ we define the intervals
$$I_n = [\PR(\lambda_n)-\delta|\al_n|^2, \PR(\lambda_n)+ \delta |\al_n|^2]\cap [-1,1]$$
and the non-negative functions
$$f_n(x) := \frac{|\al_n|^2}{|\PR(\lambda_n)-x|}\chi_{[-1,1]\setminus I_n}(x), \qquad (x\in [-1,1]),$$
where $\chi_{[-1,1]\setminus I_n}$ denotes the characteristic function of the set $[-1,1]\setminus I_n$ .

In order to estimate $\int_{[-1,1]} f_n (x) \, dx $, we observe that
for any $0<a<1$ and $b<a$,
$$\int_{a<x<1} \frac{1}{x-b}\, dx= \log(1-b)+\log \frac{1}{a-b}.$$

\noindent This along with a bit of computation yields that
\begin{equation}
\begin{split}
\int_{[-1,1]} f_n (x) \, dx &
\leq  |\al_n|^2 \left( 2\, \log \frac{1}{\delta |\al_n|^2} + \log \left |1-\PR(\lambda_n)^2 \right| \right)  \\
\\ &\leq  M |\al_n|^2\Big (1+ \log \frac{1}{|\al_n|}\Big )\label{bound}
\end{split}
\end{equation}
for some constant $M > 0$  (here we have used that $\overline{\Lambda} \subset \D$).

Now, define the function $$F(x) = \sumn f_n(x), \qquad (x\in [-1,1]).$$
Upon applying the Monotone Convergence Theorem along with \eqref{bound} and \eqref{serie logaritmo} we deduce that the integral
$$\int_{[-1,1]} F(x)dx$$
is finite. Accordingly $F(x)$ is finite for almost every $x \in [-1,1].$

If $\Omega = \bigcup_{n=1}^\infty I_n$, we observe that its Lebesgue measure satisfies
$$m(\Omega) \leq \sumn m(I_n) \leq 2 \delta \sumn |\al_n|^2 < \varepsilon.$$
Note that if $x \in [-1,1]\setminus \Omega,$ then $$ \sumn \frac{|\al_n|^2}{|\PR(\lambda_n)-x|} = F(x),$$ so $\Theta \subset \Omega \cup \{x\in[-1,1] : F(x) = \infty\}$. Since $m(\{x\in[-1,1] : F(x) = \infty \}) = 0$, we deduce that $m(\Theta) < \varepsilon$  for every arbitrary $\varepsilon>0$. Thus $m(\Theta)= 0,$ and the proof is finished.
\end{proof}

Likewise, we will make use of Lemma 4.3 in \cite{GG} which we state for the sake of completeness.

\begin{lema}\label{lema logaritmo}\cite[Lemma 4.3]{GG} Let $D_\Lambda$ be a bounded, diagonal operator associated to $\Lambda = (\lambda_n)_{n\in\N}$ with respect to an orthonormal basis $\mathcal{E}=\{e_n\}_{n\in\N}$. Assume, further, that $\Lambda$ consists of different eigenvalues of multiplicity one and $u = \sumn \al_ne_n$ is a non zero vector in $H$. Then, for almost every $x \in \R\setminus \{\PR(\lambda_n) : n \in \N\}$ the series \begin{equation}\label{convergencia log}
    \sumn (\log |\PR(\lambda_n)-x|)^2|\al_n|^2
\end{equation}
converges.
\end{lema}

With both lemmas at hand, we introduce the following definition of the \emph{relevant set} of
$\PR(\sigma(T))=\{\PR(z):\; z\in \sigma(T)\}$ for operators $T\in (\RO)$ as follows:

\begin{defi}\label{definition relevant set }
Let $T = D_\Lambda + u \otimes v \in (\RO)$ where $\Lambda = (\lambda_n)_{n\in\N}$ is bounded sequence in $\C$ and $u = \sumn \al_ne_n$ and $v = \sumn \beta_n e_n$ are non zero vectors in $H$. Denote by $$a = \min\limits_{\lambda \in \sigma(T)} \PR(\lambda), \quad \text{ and } \quad b = \max\limits_{\lambda \in \sigma(T)} \PR(\lambda).$$ We define the \emph{relevant set} of $T$ as
\begin{equation*}
\Omega(T)=  \Big \{ x \in \PR(\sigma(T))\cap (a,b) :\; \sumn \frac{|\al_n|^2}{|\PR(\lambda_n)-x|} + \frac{|\beta_n|^2}{|\PR(\lambda_n)-x|} < \infty \Big \}.
\end{equation*}
\end{defi}

\medskip

Note that the set $\Omega(T)$, if it is non-empty, consists of points of the closed set $\PR(\sigma(T))$ such that both summability conditions \eqref{lema 2} and \eqref{convergencia log} hold for the Fourier coefficients of $u$ and $v$. Likewise, observe that if $T\in (\RO)$, the spectrum $\sigma(T)$ is not reduced to a singleton. Accordingly, if $T\in (\RO)$, $\sigma(T)$ is connected and is not contained in any vertical line and, in addition,
$$\sumn |\al_n|^2 \log\frac{1}{|\al_n|} + |\beta_n|^2 \log \frac{1}{|\beta_n|}< \infty,$$
then
$\PR(\sigma(T))$ is infinite and by Lemma  \ref{lema fx} almost every point in $(a,b)$ belongs to $\Omega(T)$. Consequently, in such a case the relevant set $\Omega(T)$ has infinitely many points as well.

Next result provides a sufficient condition to ensure the existence of non zero vectors belonging to some specific spectral subspaces associated to  operators in the class $(\RO)$.

\begin{teorema}\label{proposicion}
Let $T=D_{\Lambda} + u\otimes v \in (\RO)$  where $\Lambda = (\lambda_n)_{n\in\N}$ is bounded sequence in $\C$, $u = \sumn \al_ne_n$ and $v = \sumn \beta_n e_n$ are non zero vectors in $H$. Assume that $\sigma(T)$ is connected,  does not lie in any vertical line and both $\sigma_p(T)$ and $\sigma_p(T^*)$ are empty sets.  For each $x \in \Omega(T)$ denote by
$$F_x^+ := \{\lambda \in \sigma(T) : \PR(\lambda) \geq x\},$$
$$F_x^- := \{\lambda \in \sigma(T) : \PR(\lambda) \leq x\}.$$
Suppose that there exists $\x \in \Omega(T)$ satisfying
\begin{equation}\label{ceros prop}
    1 +f_T(\x +iy) \neq 0 \quad \textnormal{for every }\x +iy \in \sigma(T).
\end{equation}
Then $H_T(F_\x^+)$ and $H_T(F_\x^-)$ are both non zero spectral subspaces.
\end{teorema}

Observe that the hypothesis on the spectra not lying on any vertical line is achievable by translating and multiplying $T$ by appropriated scalars and it is not a constraint  from the standpoint of the existence of invariant subspaces: it is imposed to ensure that $\Omega(T)$ is non-empty. The proof of Theorem \ref{proposicion} borrows some ideas from that of \cite[Theorem 4.1]{GG}.

\begin{proof} Without loss of generality, we may assume that $\overline{\Lambda} \subset \D_+ := \D\cap \{z \in \C: \textnormal{Im} \ z > 0\}$. In particular, it follows that $\sigma(T)\subset \D_+.$ This assumption is harmless regarding condition \eqref{ceros prop}.

Let $\x \in \Omega(T)$ satisfying \eqref{ceros prop} and let us show that $H_T(F_\x^+)$ is non zero (the proof for $H_T(F_\x^-)$ is similar).

First, we observe that since $\sigma(T)$ is connected, the segment $[a,b]=\PR(\sigma(T)).$ Moreover, since $\sigma(T)$ is not contained in any vertical line, $a <\x< b.$ Consequently, $F_\x^+$ is a proper, closed subset of $\sigma(T).$

Now, let $\ell_\x$ be the vertical segment passing through $\x$ with endpoints on the unit circle $\T$ and denote by
$$A_\x^+ := \{e^{i\theta} \in \T: \PR(e^{i\theta}) \geq \x \}.$$
Define $\gamma_\x^+ := \ell_\x \cup A_\x^+$. Clearly $\gamma_\x^+$ is a closed, simple, rectifiable, piecewise differentiable curve that does not intersect $\Lambda$. Moreover, $F_\x^+ = \overline{\inte(\gamma_\x^+)}\cap\sigma(T)$ so, in order to derive the desired statement, it suffices to check the conditions in \cite[Theorem 3.2]{GG} (stated also in the preliminary section).

In order to check condition (i) of \cite[Theorem 3.2]{GG}, note that $a<\x< b$ and $\sigma(T)$ is connected, so $\inte(\gamma_\x^+)$, that is, the set of points in $\mathbb{C}$ with index $1$ with respect to $\gamma_\x^+$,  intersects $\sigma(T)$.

To show condition (ii) of \cite[Theorem 3.2]{GG}, we first prove that
$$\xi \in \gamma_\x^+ \mapsto f_T(\xi) = \sumn \frac{\al_n\overline{\beta_n}}{\lambda_n-\xi}$$
is a well defined and continuous map. Observe that $f_T$ is a holomorphic function on $\C\setminus \overline{\Lambda}$ and $\gamma_\x^+$ does not intersect $\Lambda$, so we are required to show that $f_T$ is well defined and is continuous on $\gamma_\x^+ \cap \sigma(T)$. For every $\xi_\x = \x+iy \in \sigma(T)$ we have
\begin{equation*}
           | f_T(\xi_\x)|  \leq  \sumn \frac{|\al_n\beta_n|}{|\lambda_n-\xi_\x|} \leq \sumn \frac{|\al_n\beta_n|}{|\PR(\lambda_n)-\x|} \leq \left( \sumn \frac{|\al_n|^2}{|\PR(\lambda_n)-\x|} \right)^{1/2}\left( \sumn \frac{|\beta_n|^2}{|\PR(\lambda_n)-\x|} \right)^{1/2}
\end{equation*}
Since $\x$ belongs to the relevant set $\Omega(T)$, the series in the right hand side above converge. Hence,  $f_T$ is well defined and, as an application of the  Weierstrass M-test, it is continuous on $\gamma_\x^+\cap\sigma(T)$.

Now, to complete checking condition (ii), it suffices to show that $1+f_T(\xi) \neq 0$ for every $\xi \in \gamma_\x^+.$ Observe that, since $T$ has no eigenvalues, Ionascu Theorem (see \eqref{autovalores}) yields that $f_T(\xi)+1 \neq 0$ for every $\xi \in \C\setminus \overline{\Lambda}.$ Likewise, the hypothesis \eqref{ceros prop} yields that $1+f_T(\xi) \neq 0$ for every $\xi \in \gamma_x^+\cap\sigma(T)$. So, (ii) holds.

Finally, we are required to check (iii), namely
$$ \sumn \left(\int_{\gamma_\x^+} \frac{|d\xi|}{|\lambda_n-\xi|} \right)^2|\al_n|^2 < \infty.$$
The proof runs mostly as in \cite[Theorem 4.1]{GG} since $\x$ belongs to $\Omega(T)$ and hence the series
$$\sumn \left (\log |\PR(\lambda_n)-\x|\right )^2|\al_n|^2$$
converges. We sketch it, with a few modifications, for the sake of completeness.

Denote $\zeta_0=\x+i\sqrt{1-\x^2}$. Observe that
\begin{equation}\label{distance condition}
d_{+}=\inf \left\{ \sqrt{1-\x^2} + \Impart(\lambda_n):\; n\in \N \right\}> \Impart(\zeta_0)
\end{equation}
since $\overline{\Lambda}$ is strictly contained in $\D_{+}$.
Now, for each $n \in \N$,
\begin{equation}\label{suma dos}
\int_{\gamma^+_{\x}} \frac{|d\xi|}{|\lambda_n-\xi|} = \int_{\ell_{\x}} \frac{|d\xi|}{|\lambda_n-\xi|} + \int_{A_{\x}^+} \frac{|d\xi|}{|\lambda_n-\xi|}.
\end{equation}
Clearly, the second summand in the right hand side of \eqref{suma dos} is bounded by $2\pi /\textnormal{ d }(\{\lambda_n: n \in \N\},\, A_{\x}^+ ),$ which is independent of $n$. Accordingly, it suffices to bound the first summand.

A computation yields that
\begin{eqnarray}\label{integral primera}
\int_{\ell_{\x}} \frac{|d\xi|}{|\lambda_n-\xi|}& =& \displaystyle
\log \left(\frac{|\lambda_n- \zeta_0| + \Impart\left (\lambda_n-\overline{\zeta_0}\right )}
{|\lambda_n- \overline{\zeta_0}| + \Impart\left (\lambda_n-\zeta_0\right )} \right).
\end{eqnarray}
Having in mind \eqref{distance condition}  and the fact that  $\overline{\Lambda}\subsetneq \mathbb{D}$  we deduce
$$
0<d_+\leq |\lambda_n- \zeta_0| + \Impart\left (\lambda_n-\overline{\zeta_0}\right )\leq 2(1-r(D_{\Lambda}))
$$
for every $n$, where $r(D_{\Lambda})$ denotes the spectral radius of $D_{\Lambda}$. Accordingly,
\begin{equation}\label{acotacion logaritmo numerador}
\log \left(|\lambda_n- \zeta_0| + \Impart\left (\lambda_n-\overline{\zeta_0}\right )\right )
\end{equation}
remains uniformly bounded by a constant $C > 0$ for every $n$.

Likewise, as in the proof of  \cite[Theorem 4.1]{GG}, there exists a constant $A$ independent of $n$ such that
\begin{equation}\label{acotacion logaritmo}
\Big | \log \left ( |\lambda_n- \overline{\zeta_0}| + \Impart\left (\lambda_n-\zeta_0\right ) \right ) \Big | \leq
\max \left \{A,4\Big |\,\log |\Rpart( \lambda_n)-\x|\, \Big| \right \}
\end{equation}
for every $n\in \N$. Consequently, from \eqref{acotacion logaritmo numerador} and \eqref{acotacion logaritmo}, it follows that for every $n \in \N$ \begin{equation}\label{acotacion integral}
	\int_{\ell_{\x}} \frac{|d\xi|}{|\lambda_n-\xi|} \leq C + \max\{A,4|\log(|\Rpart( \lambda_n)-\x)| \},
	\end{equation}
where $C$, $A$ are positive constants independently of $n$.

\medskip

From here, it follows that
$$ \sumn \left(\int_{\ell_{\x}} \frac{|d\xi|}{|\lambda_n-\xi|}\right)^2|\al_n|^2 \leq \sumn \big (C + \max\{A,4|\log(|\Rpart( \lambda_n)-\x|)| \}\big)^2|\al_n|^2,$$
and having in mind that $\x$ belongs to $\Omega(T)$,  the series in the right hand side is convergent. Hence,
$\displaystyle \sumn\left(\int_{\gamma^+_{\x}} \frac{|d\xi|}{|\lambda_n-\xi|}\right)^2|\al_n|^2$
converges as well showing that the condition (iii) in \cite[Theorem 3.2]{GG} also holds. Thus, $H_T(\overline{\inte(\gamma_\x^+)})$ is a non zero linear manifold and therefore, $H_T(F_\x^+)$ is also non zero as we wished to show.
\end{proof}

As an application of Theorem \ref{proposicion} we have the following:

\begin{teorema}\label{teorema}
Let $T = D_\Lambda + u \otimes v \in (\RO)$, where $u = \sumn \al_ne_n$ and $v = \sumn \beta_n e_n$ are non zero vectors in $H$. Assume that there exist $\x_1$ and $\x_2$ in the relevant set $\Omega(T)$ such that $\x_1> \x_2$ and
$$1+f_T(\x_1+iy) \neq 0 \quad \textnormal{for every} \ \x_1+iy \in \sigma(T),$$
$$1+f_{T}(\x_2+iy) \neq 0 \quad \textnormal{for every} \ \x_2+iy \in \sigma(T).$$
Then, $T$ has non trivial closed hyperinvariant subspaces.
\end{teorema}

\begin{proof}
First, we may assume that $\sigma_p(T)\cup \sigma_p(T^*) = \emptyset$ and $\sigma(T)$ is connected, since otherwise $T$ has non trivial closed hyperinvariant subspaces. Likewise, since $T\in (\RO)$, the spectrum is not reduced to a singleton, and if $\sigma(T)$  is contained in a vertical line, \cite[Corollary 6.14]{RR} yields that $T$ has non trivial closed hyperinvariant subspaces. Accordingly, we may assume that $\sigma(T)$ does not lie in a vertical line and therefore $T$ meets the hypothesis of Theorem \ref{proposicion}.

Accordingly, it follows that $H_T(F_{\x_1}^+)$ is a non zero spectral subspace. In order to prove that it is non-dense, let us prove that $H_{T^*}(F_{\x_2}^-)$ is also a non zero spectral subspace for $T^*$ and Corollary \ref{adjuntos} will yield the result.
Upon applying Theorem \ref{proposicion} to $T^*$, it suffices to show that $1+ f_{T^*}(\x_2+iy) \neq 0$ for every $\x_2+iy \in \sigma(T^*)$ since $\x_2\in \Omega(T^*)$.

Note that for every $(\x_2+iy) \in \sigma(T^*)$, one has
$$\overline{1 + f_{T^*}(\x_2+iy)} = 1+ f_T(\x_2-iy).$$
Likewise, observe that $\x_2-iy \in \sigma(T)$ and by hypothesis $1+ f_T(\x_2-iy) \neq 0$. Thus, $H_{T^*}(F_{\x_2}^-)$ is a non zero linear manifold as we wished to prove.
\end{proof}

\smallskip

Last result in order to prove Theorem \ref{main result} is Theorem 2.5 in \cite{FJKP08} which deals with quasisimilar operators to operators in the class $(\RO)$. Recall that two operators $T,\, \tilde{T} \in \EL(H)$ are \textit{quasisimilar} if there exist operators $X,Y \in \EL(H)$ with $\ker X = \ker X^* = \ker Y = \ker Y^* = \{0\}$ satisfying $$ TX = X\tilde{T}, \quad YT = \tilde{T}Y.$$ Note that if $T$ and $\tilde{T}$ are quasisimilar and $T$ has a non trivial closed hyperinvariant subspace, so does $\tilde{T}$. Moreover, $\sigma_p(T)=\sigma_p(\tilde{T})$.

\begin{teorema}\label{prop quasisimilares}\cite[Theorem 2.5]{FJKP08}
Suppose $T=D_\Lambda + u \otimes v \in (\RO)$, $0 \in \Lambda'\setminus (\Lambda \cup \sigma_p(T)\cup \sigma_p(T^*)),$ and $D_\Lambda^{1/2}$ is any fixed square root of the normal operator $D_\Lambda.$ Suppose that $u$ belongs to the range of $D_\Lambda^{1/2}$. Then, $T$ and $\tilde{T}= D_\Lambda +(D_\Lambda^{-1/2}u)\otimes (D_\Lambda^{1/2})^*v$ are quasisimilar operators.
\end{teorema}

\smallskip

Observe that if $D_\Lambda \in \EL(H)$ is a diagonal operator with sequence of eigenvalues $\Lambda = \{\lambda_n\}_{n\geq 1} \subset \C\setminus (-\infty,0]$,  choosing $\sqrt{\quad } $ the principal value of the square root, the diagonal operator $D_{\sqrt{\Lambda}}$ is a square root of $D_\Lambda$. In particular, having in mind that $\Lambda^*=\{\overline{\lambda_n}: \lambda_n\in \Lambda\}$,  in this case $(D_{\sqrt{\Lambda}})^* = D_{\sqrt{\Lambda^*}}$  since $\overline{\sqrt{\lambda_n}}=\sqrt{\overline{\lambda_n}}$ for every $n\geq 1$.
%Observe that both are diagonal operators, so it is enough to show that their sequences of eigenvalues coincide. We have $$(D_\Lambda^{1/2})^* = \sumn \overline{\sqrt{\lambda_n}}e_n\otimes e_n = \sumn \sqrt{\overline{\lambda_n}}e_n\otimes e_n = (D_\Lambda^*)^{1/2}.$$

\medskip

Now, we are in position to prove Theorem \ref{main result}.

\medskip

\begin{demode}\emph{Theorem \ref{main result}.} Without loss of generality,  we may assume that $T \in (\RO)$, $\sigma(T)$ is connected not lying in a vertical line and $\sigma_p(T)\cup \sigma_p(T^*)=\emptyset.$   In such a case, $\PR(\sigma(T))$ contains infinitely many points and since
$$\sumn |\al_n|^2 \log\frac{1}{|\al_n|} + |\beta_n|^2 \log \frac{1}{|\beta_n|}< \infty$$
by hypotheses, both Lemma  \ref{lema fx} and \ref{lema logaritmo} yield that the relevant set $\Omega(T)$ has infinitely many points as well.
We will show that for every $x \in \Omega(T)$
$$1+f_T(x+iy) \neq 0 \text{ for every } x+iy \in \sigma(T)$$
and hence $T$ would meet the hypothesis of Theorem \ref{teorema} and the statement of Theorem \ref{main result} will follow.

\smallskip

Assume, by contradiction, that there exists $x_0\in \Omega(T)$ such that for some $x_0+iy_0 \in \sigma(T)$ it holds $1+f_T(x_0+iy_0) = 0.$
Let us denote  $\xi_0= x_0+iy_0 \in \sigma(T)$ and consider the operator
$$T-\xi_0 I = D_\Lambda - \xi_0 I + u\otimes v=D_{\Lambda-\xi_0} + u\otimes v$$
Note that $T-\xi_0 I \in(\RO)$ and we may assume that no eigenvalue of $D_\Lambda-\xi_0 I$ lies on $(-\infty,0]$ by multiplying $T-\xi_0 I$ by an appropriate unimodular complex number $e^{i\theta} \in \T.$

Let $\sqrt{\quad } $ denote the principal value of the square root and consider $D_{\sqrt{\Lambda - \xi_0}}$ which is a fixed square root $(D_\Lambda - \xi_0 I)^{1/2}$ of $D_\Lambda - \xi_0 I$. Since $x_0 \in \Omega(T)$ we deduce that
$$ \sumn \frac{|\al_n|^2}{|\lambda_n - \xi_0|} \leq \sumn \frac{|\al_n|^2}{|\PR(\lambda_n)-x_0|} < \infty,$$
so the vector
$$\sumn \frac{\al_n}{\sqrt{\lambda_n - \xi_0}}e_n$$
belongs to $H$. In other words, $u$ belongs to the range of $D_{\sqrt{\Lambda - \xi_0}}$. Hence, the operator $T-\xi_0 I$ meets the hypothesis of Theorem \ref{prop quasisimilares} and therefore is quasisimilar to the operator
$$S:= (D_\Lambda - \xi_0 I) + ((D_\Lambda - \xi_0 I)^{-1/2}u)\otimes ((D^*_\Lambda - \overline{\xi_0} I)^{1/2}v ).$$
By taking adjoints,  $T^* - \overline{\xi_0}I$ is quasisimilar to
$$S^* = (D_\Lambda^* - \overline{\xi_0}I) +  ((D^*_\Lambda - \overline{\xi_0} I)^{1/2}v )\otimes ((D_\Lambda - \xi_0 I)^{-1/2}u).$$
Now, the goal will be proving that $0$ is an eigenvalue of $S^*$ upon applying Ionascu cha\-rac\-te\-rization \cite{I01} (see the statement in the preliminary section). As a consequence, $T^*- \overline{\xi_0}I$ will have 0 as an eigenvalue as well which will lead to a contradiction.

\smallskip

First, observe that $0\notin \Lambda^* - \overline{\xi_0}$ since $\xi_0=x_0+i y_0 \notin \Lambda$ because $x_0\in\Omega(T)$.

Secondly,  in order to show that $(D^*_\Lambda - \overline{\xi_0} I)^{1/2}v$ belongs to the range of $D^*_\Lambda - \overline{\xi_0}I$, note that
$$(D^*_\Lambda - \overline{\xi_0} I)^{1/2} v = \sumn (\overline{\lambda_n}-\overline{\xi_0})^{1/2}\beta_n e_n.$$
Hence,
$$ \sumn \frac{|\lambda_n - \xi_0||\beta_n|^2}{|\lambda_n-\xi_0|^2} = \sumn \frac{|\beta_n|^2}{|\lambda_n-\xi_0|} \leq  \sumn \frac{|\beta_n|^2}{|\PR(\lambda_n)-x_0|}$$
which converges since $x_0 \in \Omega(T)$.  Thus, $(D^*_\Lambda - \overline{\xi_0} I)^{1/2} v$ belongs to the range of $D^*_\Lambda - \overline{\xi_0}I$.

\smallskip

Finally, let us show that $1+f_{S^*}(0) = 0$. As
$$(D_\Lambda - \xi_0 I)^{-1/2}u = \sumn (\lambda_n - \xi_0)^{-1/2}\al_ne_n,$$
we have \begin{equation*}
    \begin{split}
       1+ f_{S^*}(0) &= 1+\sumn \frac{(\overline{\lambda_n}-\overline{\xi_0})^{1/2}\beta_n (\overline{\lambda_n} - \overline{\xi})^{-1/2}\overline{\al_n}}{\overline{\lambda_n}-\overline{\xi_0} } =1+ \sumn \frac{\beta_n\overline{\al_n}}{\overline{\lambda_n}-\overline{\xi_0} } \\ & = \overline{1+f_T(\xi_0)} = 0.
    \end{split}
\end{equation*}
Thus, $0$ is an eigenvalue of $S^*$ as we claimed which yields the contradiction.

\smallskip

Accordingly, $f_T(x+iy) +1 \neq 0$ for every $x \in \Omega(T)$ and $y \in \R$ such that $x+iy \in \sigma(T)$. Theorem \ref{teorema} provides a non trivial hyperinvariant subspace for $T$ as we wished.
\end{demode}

A careful reading of the proof of Theorem \ref{main result} shows that the summability condition \eqref{sumabilidad} plays a role in order to ensure the existence of \emph{appropriate points} in the relevant set $\Omega(T)$. As a consequence, a sort of more general statement may be stated:

\begin{coro}\label{corolario}
Let $T=D_{\Lambda} + u\otimes v \in (\RO)$ where $\Lambda = (\lambda_n)_{n\in\N}$ is bounded sequence in $\C$, $u = \sumn \al_ne_n$ and $v = \sumn \beta_n e_n$ are non zero vectors in $H$. Assume that $\sigma(T)$ is connected,  does not lie in any vertical line and both $\sigma_p(T)$ and $\sigma_p(T^*)$ are empty sets. Let $[a,b] = \{\Rpart(\lambda): \lambda \in \sigma(T)\}$ and assume there exist $a<x_1<x_2<b$ such that
\begin{equation}\label{exponente 1}
     \sumn \frac{|\al_n|^2}{|\PR(\lambda_n)-x_i|} + \frac{|\beta_n|^2}{|\PR(\lambda_n)-x_i|} < \infty,
\end{equation}
%\begin{equation}\label{serie logaritmo-}
%    \sumn \left (\log |\PR(\lambda_n)-x_i|\right )^2|\al_n|^2+ \left( \log |\PR(\lambda_n)-x_i|\right )^2|\beta_n|^2 < \infty
%\end{equation}
for $i=1,2.$ Then, $T$ has non trivial closed hyperinvariant subspaces.
\end{coro}

\section{The obstruction of Theorem \ref{main result}}\label{sec3}

A natural question which arises having Corollary \ref{corolario} at hands is if it is always possible to find $x_1$ and $x_2$ with $a<x_1<x_2<b$ such that the summability condition  \eqref{exponente 1} holds. Observe that if this is the case, due to the discussed reductions,  every rank-one perturbation of a diagonalizable normal operator would have non trivial closed hyperinvariant subspaces.

\smallskip

%Clearly Lemma \ref{lema logaritmo} yields that almost every point in $(a,b)$ satisfies \eqref{serie logaritmo-}. So, t
%Note that the question is reduced to study the set of points $x\in \PR(\sigma(T))$ such that \eqref{exponente 1} is satisfied.
Next, we provide an example showing that such a set could be empty. In particular, we exhibit a bounded sequence of real numbers $(r_n)_{n\in \N}$ dense in $[-1,1]$ and a non zero vector $u$ in the classical $\ell^2$ space with Fourier coefficients $\{\alpha_n\}_{n\geq 1}$ with respect to the canonical basis $\mathcal{E}=\{e_n\}_{n\geq 1}$ satisfying
$$ \sumn \frac{|\al_n|^2}{|r_n-x|} = \infty$$ for every $x \in (-1,1)$.

\smallskip

Note that, in such a case, if $\Lambda=\{\lambda_n\}_{n\geq 1}$ is any bounded sequence in $\C$ such that $\PR(\Lambda)$ contains $\{r_n\}_{n\geq 1}$ and $D_\Lambda$ is the corresponding diagonal operator with respect to $\mathcal{E}$ acting on $\ell^2$ then any rank-one perturbation $T=D_\Lambda + u\otimes v\in \EL(\ell^2)$, with $v\in \ell^2\setminus \{0\}$,  no longer meets the hypotheses of Corollary \ref{corolario}. Roughly speaking, this example shows the blue obstruction of the discussed approach in order to provide non trivial closed hyperinvariant subspaces.

Likewise, observe that by taking $\Lambda=\{r_n\}_{n\geq 1}$ and $u, v\in \ell^2\setminus \{0\}$ such that their Fourier coefficients $\{\alpha_n\}_{n\geq 1}$ and $\{\beta_n\}_{n\geq 1}$ satisfy
$$
\sum_{n=1}^{\infty} \left | \frac{\al_n\, \overline{\beta_n}}{r_n-x} \right | = \infty
$$
for every $x \in (-1,1)$, the operators $T=D_\Lambda + u\otimes v\in \EL(\ell^2)$ are examples of rank-one perturbations of diagonal operators such that their Borel series $f_T$ are not absolutely convergent at any point in $\sigma(T)$. Nevertheless, a theorem of Radjabalipour and Radjavi \cite{Radjabalipour-Radjavi} yields that such $T$'s are indeed decomposable, and hence, have a rich lattice of invariant subspaces. Accordingly, these examples clearly show an edge in order to apply previous approaches to produce invariant subspaces since they rely on the finiteness of associated Borel series to the operator.

\smallskip

In what follows, we will make use of some ideas by Stampfli \cite{Stampfli}. We thank Prof. Yakubovich for the discussions regarding next example.

\smallskip

\begin{ej}
Let us consider a dyadic partition of the natural numbers $\N$ as follows:  for each non-negative integer $m\in \N_0$, let the set $S_m$ be
$$S_m = \{2^m+k : 0 \leq k \leq 2^m-1 \}.$$
It is clear that $\N= \bigsqcup_{m=0}^\infty S_m.$ Now, define the sequence $\{r_n\}_{n\in\N}$ recursively taking into account the defined partition as follows: for each $n\in S_m$, if $n=2^m+k$ for $m\in \N_0$ and $0\leq k \leq 2^m-1$
$$r_n=r_{2^m+k}:= -1 + \frac{1}{2^m} + \frac{k}{2^{m-1}}.$$

Observe that $\{r_n\}_{n\in\N} $ is a well defined sequence contained in $(-1,1)$ with particular values:
$$
r_1=0;\; r_2=-\frac{1}{2};\; r_3=\frac{1}{2};\; r_4=-\frac{3}{4};\; r_5=-\frac{1}{4};\; r_6=\frac{1}{4};\;  r_7=\frac{3}{4};\; \cdots
$$
Likewise, $\{r_n\}_{n\in\N} $ is dense in $[-1,1].$ Let $\{\gamma_n\}_{n\geq 0} \subset \C$ be a sequence such that
$$ \sumno 2^n|\gamma_n|^2 < \infty \quad \textnormal{and} \quad \sumno n\, 2^n|\gamma_n|^2 = \infty.$$
This can be accomplished by taking, for instance, $\gamma_0=\gamma_1 = 0$ and $\gamma_n = \frac{1}{2^{n/2}\sqrt{n}\log(n)}$ for every $n\geq 2.$

Now, let us consider the vector $u = \sumn \al_ne_n$ defined recursively as follows: for each $n\in S_m$, with $m\in \N_0$, let
$$\al_{n} = \gamma_m.$$
% \qquad (n=2^m+k \text{ with } m\geq 0 and \ 0\leq k \leq 2^m-1.
Note that for each $m \in \N_0$, by definition,  $\al_n$ takes the same value for every $n \in S_m$. Observe that $u \in \ell^2$ since
$$\sumn |\al_n|^2 = \sum_{m=0}^\infty \sum_{k=0}^{2^m-1} |\gamma_m|^2 = \sum_{m=0}^\infty 2^m |\gamma_m|^2 < \infty.$$

Now, we show that
$$ \varphi(x):=\sumn \frac{|\al_n|^2}{|r_n-x|} = \infty$$
for every $x\in (-1,1).$

Let $-1<x\leq 0$; the case $0<x <1$ is analogous (indeed,  $\varphi(-x)=\varphi(x)$). There exists a positive integer $N \in \N$ such that for every $n\geq N$
$$-1 + \frac{1}{2^n} < x.$$
For each $n\geq N$, there exists a non-negative integer $k_{0}(n)$ depending on $n$, with $1<k_{0}(n)\leq 2^{n-1}$, such that
$$-1 + \frac{1}{2^n}+ \frac{k_{0}(n)-1}{2^{n-1}} \leq x < -1 + \frac{1}{2^n}+ \frac{k_{0}(n)}{2^{n-1}}$$
or equivalently
$$r_{2^n+k_{0}(n)-1} \leq x < r_{2^n+k_{0}(n)}.$$
Note that for each $k=k_{0}(n),\, k_{0}(n)+1,\, k_{0}(n)+2, \cdots,  k_{0}(n)+2^{n-1}-1$,
$$
\frac{k_{0}(n)-1-k}{2^{n-1}}\leq x- r_{2^n+k}< \frac{k_{0}(n)-k}{2^{n-1}},
$$
so
\begin{equation*}\label{estimacion yakubovich}
|x-r_{2^n+k}| \leq \frac{k-k_{0}(n)+1}{2^{n-1}}.
\end{equation*}
Then
\begin{equation*}
\begin{split}
\sumn \frac{|\al_n|^2}{|r_n-x|} & = \sum_{m=0}^\infty \sum_{k=0}^{2^m-1} \frac{|\gamma_m|^2}{|x-r_{2^m+k}|} \\&
\geq  \sum_{m=N}^\infty \, \sum_{k=k_{0(m)}}^{k_{0}(m)+2^{m-1}-1} \frac{|\gamma_m|^2}{k-k_{0}(m)+1}2^{m-1} \\ &
=  \sum_{m=N}^\infty |\gamma_m|^2 \, 2^{m-1} \, \sum_{k=1}^{2^{m-1}-1} \frac{1}{k} \\&
>  \sum_{m=N}^\infty |\gamma_m|^2\, 2^{m-1} \log(2^{m-1}) \\ &
= \log(2) \sum_{m=M}^\infty |\gamma_m|^2\, 2^{m-1}\cdot (m-1),
\end{split}
\end{equation*}
which diverges because of the choice of the sequence $\{\gamma_m\}_{m\geq 0}$. This concludes the example.
\end{ej}

\section{Finite rank perturbations of diagonalizable normal operators}\label{sec4}

In this section, we consider finite rank perturbations of diagonal operators, namely,
\begin{equation*}\label{rango finito}
     T = D_\Lambda + \sumk u_k\otimes v_k \in \EL(H),
\end{equation*}
where, as before, $\Lambda = \{\lambda_n\}_{n\geq 1}\subset\C$ is a bounded sequence of complex numbers, $D_\Lambda$ is a diagonal operator with respect to an orthonormal basis $\mathcal{E}=\{e_n\}_{n\geq 1}$ with eigenvalues $\Lambda$ and $u_k$ and $v_k$ are non zero vectors in $H$ for $1\leq k \leq N,$ where $N\geq 2$ is fixed. We will denote the Fourier coefficients of the vectors $u_k, v_k$ with respect to $\mathcal{E}$ as
$$ u_k = \sumn \al_n^{(k)}e_n, \quad v_k = \sumn \beta_n^{(k)}e_n \qquad (1\leq k \leq N).$$
Our main goal in this section is generalizing Theorem \ref{main result} in this setting proving the following result:

\begin{teorema}\label{resultado rango finito}
Let $T = D_\Lambda + \sumk u_k\otimes v_k \in \EL(H)\setminus \mathbb{C}\,  Id_{H}$ be any finite rank perturbation of a diagonal normal operator $D_{\Lambda}$ with respect to an orthonormal basis $\mathcal{E}=\{e_n\}_{n\geq 1}$ where $u_k = \sumn \al_n^{(k)}e_n$ and $v_k = \sumn \beta_n^{(k)}e_n$ are non zero vectors in $H$. Then $T$ has non trivial closed hyperinvariant subspaces provided that for each $1\leq k \leq N$
\begin{equation}\label{sumabilidad-1}
 \sum_{n\in \mathcal{N}_{u_k}} \left |\al_n^{(k)}\right |^2 \log \frac{1}{\left |\al_n^{(k)}\right |} + \sum_{n\in \mathcal{N}_{v_k}} \left |\beta_n^{(k)}\right |^2 \log \frac{1}{\left |\beta_n^{(k)}\right |}< \infty, 
\end{equation}
where
\begin{equation}
\mathcal{N}_{u_k}=\{n\in \mathbb{N}:\; \al_n^{(k)}\neq 0\}, \quad \mathcal{N}_{v_k}=\{n\in \mathbb{N}:\; \beta_n^{(k)}\neq 0 \}.\label{index set}
\end{equation}
\end{teorema}

The proof of Theorem \ref{resultado rango finito} consists of an appropriate generalization of the approach carried over for rank-one perturbations of diagonal operators. Nevertheless, there are some technical issues which are not straightforward and  we will mainly address them instead of repeating  again some of the ideas.\\

We start by generalizing the class $(\RO)$ in the setting of finite rank perturbations of diagonal operators:
\\

\begin{defi}
With the notation as introduced above, the operator $T = D_\Lambda + \sumk u_k\otimes v_k \in \EL(H)\setminus \mathbb{C}\,  Id_{H}$ belongs to the class $(\RO)_N$ if:
\begin{enumerate}
    \item [(i)] For each $n \in \N$ there exist $1\leq k_1,k_2 \leq N$ such that $\al_n^{(k_1)}\beta_n^{(k_2)} \neq 0.$
    \item [(ii)] Each $\lambda \in \Lambda$ has multiplicity less or equal than $N$.
    \item [(iii)] The derive set  $\Lambda'$ is not reduced to a singleton.
\end{enumerate}
\end{defi}

As it was shown in \cite[Proposition 6.1]{GG}, if $T=D_\Lambda + \sumk u_k\otimes v_k \in \EL(H)\setminus \mathbb{C}\,  Id_{H}$  with $u_k, \,v_k\in H$ for $1\leq k\leq N$ does not belong to the corresponding class $(\RO)_N$, then it has non trivial closed hyperinvariant subspaces. Accordingly, from the standpoint of view of studying invariant subspaces, we restrict ourselves to this particular subclass of operators.\\

We begin with an appropriate generalization of Borel series associated to finite rank perturbations of diagonal operators.

\medskip

Let $T=D_\Lambda + \sumk u_k\otimes v_k\in \EL(H)\setminus \mathbb{C}\,  Id_{H}$ be an operator in $(\RO)_N$, where $\Lambda=\{\lambda_n\}_{n\geq 1}$ is a bounded sequence in $\mathbb{C}$.  For each $i,j \in \{1,\cdots, N\}$, let us define
$$f_T^{(i,j)}(z) = \sumn \frac{\al_n^{(i)}\overline{\beta_n^{(j)}}}{\lambda_n-z}$$
for those complex numbers $z$ such that the series converges. Clearly, each $f_T^{(i,j)}$ is an analytic function in $\C\setminus \overline{\Lambda}$. Let $M_T$ be the $N\times N$ matrix function
$$ M_T(z)= \left( \begin{matrix}
f_T^{(1,1)}(z)+1 & f_T^{(1,2)}(z) & \cdots & f_T^{(1,N)}(z) \\
f_T^{(2,1)}(z) & f_T^{(2,2)}(z)+1 & \cdots & f_T^{(2,N)}(z) \\
 \vdots & \vdots & \vdots & \vdots\\
 f_T^{(N,1)}(z) & f_T^{(N,2)}(z) &\cdots& f_T^{(N,N)}(z)+1 \end{matrix}\right)$$
for those $z\in \mathbb{C}$ such that $f_T^{(i,j)}(z)$ is defined for every $i,j \in \{1,\cdots, N\}$. This matrix function plays a significant role regarding the existence
of eigenvalues of the operator $T$:\\

 \begin{lema}\rm{(\cite[Lemma 3.1]{FX12})} Let $T = D_\Lambda + \sumk u_k\otimes v_k \in \EL(H)\setminus \mathbb{C}\,  Id_{H}$ be a finite rank perturbation of a diagonal normal operator $D_{\Lambda}$ with respect to an orthonormal basis $\mathcal{E}=\{e_n\}_{n\geq 1}$ where $u_k = \sumn \al_n^{(k)}e_n$ and $v_k = \sumn \beta_n^{(k)}e_n$ are non zero vectors in $H$ for $1\leq k\leq N$. Assume further  that $T\in (\RO)_N$ and $z \in \C$ such that
 \begin{enumerate} \label{autovalores rango finito}
     \item [(i)] $z \notin \Lambda.$
     \item [(ii)] For every $1\leq k \leq N$,
     $$\sumn \frac{\left |\al_n^{(k)}\right |^2}{|\lambda_n-z|^2} < \infty.$$
     \item [(iii)] $\ker(T-zI) = \{0\}.$
 \end{enumerate}
 Then $M_T(z)$ is invertible.
 \end{lema}

We stress here that \cite[Lemma 3.1]{FX12} also asks that the series
$$\sumn \frac{\left |\beta_n^{(k)}\right |^2}{|\lambda_n-z|^2} < \infty,$$
converge for every $1\leq k \leq N$, but a careful reading of the proof shows that this condition is no longer necessary to yield the conclusion. \\

Next result, which we borrow from \cite{GG}, allows us to exhibit non zero spectral subspaces of operators $T \in (\RO)_N.$

\begin{teorema}{\rm (\cite[Theorem 6.8]{GG})} Let $T=D_\Lambda + \sumk u_k\otimes v_k \in (\RO)_N$ with $u=\sumn \al_n^{(k)}e_n$ and $v_k = \sumn \beta_n^{(k)}e_n$ non zero vectors in $H$ for each $1\leq k\leq N.$ Assume $\sigma_p(T)\cup\sigma_p(T^*) = \emptyset$ and $\sigma(T)$ is connected. Suppose, further, that there exists a closed, simple, piecewise differentiable curve $\gamma$ in $\C$ not intersecting $\Lambda$ such that
\begin{enumerate} \label{existencia rango finito}
    \item [(a)] $\sigma(T)\cap\inte(\gamma) \neq \emptyset$ and $\inte(\gamma)\cap \rho(T) \neq \emptyset.$
    \item [(b)] The maps $$ \xi \in \gamma \mapsto f_T^{(i,j)}(\xi),  \quad (1\leq i,j\leq N)$$
    and
    $$\xi \in \gamma \mapsto \frac{1}{\det(M_T(\xi))}$$
    are well defined and continuous on $\gamma.$
    \item[(c)] The series
    $$\sumn \left( \int_\gamma \frac{|d\xi|}{|\lambda_n-\xi|}\right)^2|\al_n^{(k)}|^2 < \infty,$$
    converge for every $1\leq k \leq N.$
\end{enumerate}
    Then, $H_T(\overline{\inte(\gamma)})$ is a non zero spectral subspace.
\end{teorema}

Next definition is the generalization of Definition \ref{definition relevant set } for operators $T\in (\RO)_N$:

\begin{defi}
Let $T = D_\Lambda + \sumk u_k\otimes v_k \in \EL(H)\setminus \mathbb{C}\,  Id_{H}$ be a finite rank perturbation of a diagonal normal operator $D_{\Lambda}$ with respect to an orthonormal basis $\mathcal{E}=\{e_n\}_{n\geq 1}$, where $u_k = \sumn \al_n^{(k)}e_n$ and $v_k = \sumn \beta_n^{(k)}e_n$ are non zero vectors in $H$ for $1\leq k\leq N$. Let
$$a = \min\limits_{\lambda \in \sigma(T)} \PR(\lambda), \quad \text{ and } \quad b = \max\limits_{\lambda \in \sigma(T)} \PR(\lambda).$$
The \emph{relevant set} of $T$ consists of
\begin{equation*}
\Omega_N(T)=  \Big \{ x \in \PR(\sigma(T))\cap (a,b) : \sumn \frac{\left |\al_n^{(k)}\right |^2}{|\PR(\lambda_n)-x|} + \frac{\left |\beta_n^{(k)}\right |^2}{|\PR(\lambda_n)-x|} < \infty, \quad \textnormal{for every} \ 1\leq k \leq N \Big \}.
\end{equation*}
\end{defi}

\medskip

Note that if $u_k$ and $v_k$ have non zero Fourier coefficients for every $1\leq k\leq N$, the index sets in \eqref{index set} coincides with $\mathbb{N}$. Hence, if \eqref{sumabilidad-1} holds for every $1\leq k\leq N$, Lemma \ref{lema fx} yields that almost every point in $(a,b)$ belongs to $\Omega_N(T).$

\medskip

On the other hand, if $u_k$ (or $v_k$) has a Fourier coefficient which is zero for some $1\leq k\leq N$, say $\alpha_{n_0}^{(k)}$ for instance, though the index set $\mathcal{N}_{u_k}\subsetneq \mathbb{N}$, the conclusion of Lemma \ref{lema fx} clearly holds if \eqref{serie logaritmo} is replaced by
$$\sum_{n\in \mathcal{N}_{u_k}}  |\al_n^{(k)}|^2\, \log \frac{1}{\left |\al_n^{(k)}\right |} < \infty.$$
Accordingly, in such a case, under the hypotheses \eqref{sumabilidad-1}, it also follows that almost every point in $(a,b)$ belongs to $\Omega_N(T).$

Thus, in both cases, if $T \in (\RO)_N$ and $\sigma(T)$ is a connected set not contained in any vertical line in the complex plane then $\Omega_N(T)$ contains infinitely many points.

\medskip

Next result generalizes Theorem \ref{proposicion} to this context:

\begin{teorema}
Let $T = D_\Lambda + \sumk u_k\otimes v_k \in \EL(H)\setminus \mathbb{C}\,  Id_{H}$ be a finite rank perturbation of a diagonal normal operator $D_{\Lambda}$ with respect to an orthonormal basis $\mathcal{E}=\{e_n\}_{n\geq 1}$, where $u_k = \sumn \al_n^{(k)}e_n$ and $v_k = \sumn \beta_n^{(k)}e_n$ are non zero vectors in $H$ for $1\leq k\leq N$. Assume $T \in (\RO)_N$, $\sigma(T)$  is a connected set not contained in any vertical line in $\mathbb{C}$ and $\sigma_p(T)\cup \sigma_p(T^*)=\emptyset$. For $\x \in \Omega_N(T)$ let $F_\x^+$ and $F_\x^-$ denote the sets
$$F_\x^+ = \{ \lambda \in \sigma(T): \PR(\lambda) \geq \x\},$$
$$F_\x^- = \{ \lambda \in \sigma(T) : \PR(\lambda) \leq \x\}.$$
If $\x \in \Omega_N(T)$ satisfies
\begin{equation}\label{ceros rango finito}
    \det(M_T(\x+iy))\neq 0 \quad \textnormal{for every} \ \x+iy \in \sigma(T),
\end{equation}
then $H_T(F_\x^+)$ and $H_T(F_\x^-)$ are non zero spectral subspaces.
\end{teorema}

The proof follows the lines of Theorem \ref{proposicion} and hence, we focus on those details that differs from it.

\begin{proof}
Without loss of generality, we assume that $\overline{\Lambda} \subset \D_+$ and prove that $H_T(F^+)$ is non zero (the argument for $H_T(F^-)$ is analogous).

Taking $\x \in \Omega_N(T)$ satisfying \eqref{ceros rango finito} and arguing as is Theorem \eqref{proposicion}, it suffices to show that $H_T(\overline{\inte(\gamma_\x^+)})$ is non zero by checking that conditions (a),(b) and (c) of Theorem \ref{existencia rango finito} are fulfilled.
Conditions (a) and (c) follows  as in Theorem \ref{proposicion}, so we are left to prove that condition (b) also holds.

First, let us show that the maps
$$ \xi \in \gamma_\x^+ \mapsto f_T^{(i,j)}(\xi) = \sumn \frac{\al_n^{(i)}\overline{\beta_n^{(j)}}}{\lambda_n-\xi} $$
\noindent are well defined and are continuous for $1\leq i,j\leq N.$ Since these functions are holomorphic on $\C\setminus\overline{\Lambda},$ it suffices to prove that they are well defined and are continuous in $\gamma_\x^+\cap \sigma(T)$.

Note that for every $\xi = \x+iy \in \sigma(T)$
$$|f_T^{(i,j)}(\xi)| \leq \sumn \left | \frac{\al_n^{(i)}\overline{\beta_n^{(j)}}}{\lambda_n-\xi} \right  |
\leq
\sumn \left | \frac{\al_n^{(i)}\overline{\beta_n^{(j)}}}{\PR(\lambda_n)-\x} \right |
\leq \left( \sumn \frac{\left |\al_n^{(i)}\right |^2}{|\PR(\lambda_n)-\x|} \right)^{1/2}\left( \sumn \frac{\left |\beta^{(j)}_n\right |^2}{|\PR(\lambda_n)-\x|} \right)^{1/2}\hspace*{-0,2cm}<\infty
$$
since $\x \in \Omega_N(T)$. This shows that each $f_T^{(i,j)}$ is well defined and is continuous on $\gamma_\x^+$.

In order to prove that the map $$\xi \in \gamma_\x^+ \mapsto \frac{1}{\det(M_T(\xi))}$$ is well defined and is continuous, observe that  $\xi \in \gamma_\x^+ \mapsto \det(M_T(\xi))$ is well defined and it is continuous because each $f_T^{(i,j)}$ is. Likewise,  $\det(M_T(\xi)) \neq 0$ for every $\xi \in \gamma_\x^+$. Indeed, this latter fact follows for those $\xi \in \gamma_\x^+ \cap \sigma(T)$ by hypothesis (equation \eqref{ceros rango finito}), while for those $\xi \in \gamma_\x^+ \cap \rho(T)$ from Lemma \ref{autovalores rango finito}, since $\sigma_p(T)$ is empty and  the series
$$ \sumn \frac{|\al_n^{(k)}|^2}{|\lambda_n-z|^2}$$
is convergent for each $1\leq k \leq N$ and   $z \in \C\setminus \overline{\Lambda}$.  Accordingly, condition (b) in Theorem \ref{existencia rango finito} is fulfilled and hence $H_T(\overline{\inte(\gamma_\x^+)})$ is non zero as wished to prove.
\end{proof}

\begin{teorema}\label{teorema ceros rango finito} Let $T=D_\Lambda + \sumk u_k\otimes v_k \in (\RO)_N$, where $u=\sumn \al_n^{(k)}e_n$ and $v_k = \sumn \beta_n^{(k)}e_n$ are non zero vectors in $H$ for each $1\leq k\leq N$. Assume that there exist $\x_1$ and $\x_2$ in the relevant set $\Omega(T)$ such that $\x_1> \x_2$ and
$$\det(M_T(\x_1+iy)) \neq 0 \quad \textnormal{for every}\ \x_1+iy \in \sigma(T),$$
$$ \det(M_T(\x_2+iy)) \neq 0 \quad \textnormal{for every}\ \x_2+iy \in \sigma(T).$$
Then, $T$ has a non trivial closed hyperinvariant subspace.
\end{teorema}

The proof of Theorem \ref{teorema ceros rango finito} follows the lines of that of Theorem \ref{teorema} having in mind that for any $\x_2+iy \in \sigma(T^*)$
$$\det(M_T^*(\x_2+iy)) = \det(M_T(\x_2-iy))$$
where $\x_2-iy \in \sigma(T)$. \\

The last ingredient in order to prove Theorem \ref{resultado rango finito} is a result in the spirit of
\cite[Theorem 2.5]{FJKP08} regarding quasisimilar operators in the class $(\RO)_N$.

\begin{teorema}\label{quasisimilares rango finito} Let $T=D_\Lambda + \sumk u_k\otimes v_k \in (\RO)_N$, where $u=\sumn \al_n^{(k)}e_n$ and $v_k = \sumn \beta_n^{(k)}e_n$ are non zero vectors in $H$ for each $1\leq k\leq N$. Suppose that $0 \in \Lambda' \setminus (\Lambda\cup\sigma_p(T)\cup\sigma_p(T^*))$ and let $D_\Lambda^{1/2}$ be any fixed square root of the normal operator $D_\Lambda.$ Assume that $u_k$ belongs to the range of $D_\Lambda^{1/2}$ for every $1\leq k\leq N.$ Then, $T$ and $\tilde{T}= D_\Lambda + \sumk (D_\Lambda^{-1/2}u_k)\otimes (D_\Lambda^{1/2})^*v_k$ are quasisimilar.
\end{teorema}

The proof is similar to the one of \cite[Theorem 2.5]{FJKP08} and is included for the sake of completeness.

\begin{proof} By hypotheses, both operators $D_\Lambda^{1/2}$ and $T$ are quasiaffinities, namely,
$$\ker(D_\Lambda^{1/2}) = \ker((D_\Lambda^{1/2})^*)=\ker(T)=\ker(T^*) = \{0\}.$$
Likewise, $U:=D_\Lambda^{-1/2}T = D_\Lambda^{1/2}+ \sumk(D_\Lambda^{-1/2}u_k)\otimes v_k$ is a bounded operator being a quasiaffinity.

Observe that
$$ T= D_\Lambda^{1/2}U, \quad \tilde{T}= UD_\Lambda^{1/2},$$
so
$$TD_\Lambda^{1/2}= D_\Lambda^{1/2}UD_\Lambda^{1/2} = D_\Lambda^{1/2}\tilde{T}$$
and
$$UT = UD^{1/2}U = \tilde{T}U.$$
Consequently, $T$ and $\tilde{T}$ are quasisimilar operators and the statement follows.
\end{proof}

\medskip

Finally we are in position to prove Theorem \ref{resultado rango finito} as a byproduct of theorems  \ref{teorema ceros rango finito} and \ref{quasisimilares rango finito}.\\

\medskip

\noindent \textit{Proof of Theorem \ref{resultado rango finito}.} Without loss of generality, we may assume that $T \in (\RO)_N$, $\sigma(T)$ is a connected set not lying in any vertical line of $\mathbb{C}$ and $\sigma_p(T)\cup \sigma_p(T^*)=\emptyset$. Then \eqref{sumabilidad-1} along with lemmas \ref{lema logaritmo} and \ref{lema fx} yield that the relevant set $\Omega_N(T)$ contains infinitely many points.

Let $\x \in \Omega_N(T)$ be fixed and let us show that $\det(M_T(\x+iy)) \neq 0$ for every $\x+iy \in \sigma(T)$ following some of the ideas in the proof of the Theorem \ref{main result}.

Assume, by contradiction, that there exists $\xi = \x+iy \in \sigma(T)$ such that $\det(M_T(\xi)) = 0$ and consider the operator
$$T-\xi I = (D_\Lambda-\xi I) + \sumk u_k \otimes v_k,$$
which belongs to $(\RO)_N$.  As in the proof of Theorem \ref{main result}, we may assume that no eigenvalues of $D_\Lambda-\xi I$ lies on $(-\infty,0]$. In particular, denoting $\sqrt{\quad } $ the principal value of the square root, we fix $D_{\sqrt{\Lambda - \xi}}$ as a square root of $D_\lambda - \xi I$, denoted by $(D_\Lambda-\xi I)^{1/2}$.
Since $\x \in \Omega_N(T)$
$$ \sumn \frac{\left |\al_n^{(k)}\right |^2}{|\lambda_n-\xi|} \leq \sumn \frac{|\al_n^{(k)}|^2}{|\PR(\lambda_n)-\x|} < \infty, \qquad (1\leq k\leq N),$$
so the vectors $u_k$ belongs to the range of $(D_\Lambda-\xi I)^{1/2}$ for $1\leq k \leq N.$ Upon applying Theorem \ref{quasisimilares rango finito} it follows that $T-\xi I$ and
$$S := (D_\Lambda-\xi I) + \sumk (D_\Lambda-\xi I)^{-1/2}u_k\otimes (D_\Lambda^*-\overline{\xi}I)^{1/2}v_k$$
are quasisimilar operators. Then, $T^*-\overline{\xi}I$ and $S^*= D_\Lambda^*-\overline{\xi}I + \sumk (D_\Lambda^*-\overline{\xi}I)^{1/2}v_k\otimes  (D_\Lambda-\xi I)^{-1/2}u_k  $ are also quasisimilar what implies, in particular, that $\sigma_p(S^*) = \emptyset$.

On the other hand, since $\x \in \Omega_N(T)$ the vectors  $(D_\Lambda^*-\overline{\xi}I)^{1/2}v_k$ belong to the range of $(D_\Lambda^*-\overline{\xi}I)$ and hence, by Lemma \ref{autovalores rango finito}, the determinant $\det(M_{S^*}(0)) \neq 0$. Note that for any $1\leq i,j\leq N$
\begin{equation*}
    \begin{split}
        f_{S^*}^{(i,j)}(0) = \sumn \frac{(\overline{\lambda_n - \xi})^{1/2}\beta_n^{(i)}(\overline{\lambda_n-\xi})^{-1/2}\overline{\al_n^{(j)}}}{\overline{\lambda_n}-\overline{\xi}} = \sumn \frac{\beta_n^{(i)}\overline{\al_n^{(j)}}}{\overline{\lambda_n}-\overline{\xi}} = \overline{f_T^{(j,i)}(\xi)},
    \end{split}
\end{equation*}
so $M_{S^*}(0) = M_T^*(\xi)$. By assumption, $\det(M_T(\xi)) = 0$ so $\det(M_{S^*}(0)) = 0$ as well, what yields a contradiction.

Accordingly, $\det(M_T(\x+iy)) \neq 0$ for every $\x+iy\in \sigma(T)$ with $\x \in \Omega_N(T)$. Upon applying Theorem \ref{teorema ceros rango finito} the proof of Theorem \ref{resultado rango finito} follows. \hfill $\Box$

	\end{document}